\documentclass[12pt,oneside]{article}
\usepackage{amsmath,amssymb,amsfonts,amsthm}
\usepackage{color}
\newcommand{\T}{{\cal T}}

\newcommand{\Real}{\mathbb R}

\newcommand{\set}[1]{\left\{#1\right\}}

\newcommand{\To}{\longrightarrow}

\newcommand{\p}{\pi^{-1}(TM)}

\def\Section#1{\vspace{30truept}\addtocounter{section}{1}\setcounter{thm}{0}\setcounter{equation}{0}
{\noindent\Large\bf\arabic{section}.~~#1}\par \vspace{12pt}}
\newtheorem{thm}{Theorem}[section]
\newtheorem{cor}[thm]{Corollary}
\newtheorem{lem}[thm]{Lemma}
\newtheorem{prop}[thm]{Proposition}
\newtheorem{defn}[thm]{Definition}

\newtheorem{rem}[thm]{Remark}

\numberwithin{equation}{section}
\begin{document}
\title{\bf  INTRINSIC THEORY\\ OF PROJECTIVE CHANGES IN FINSLER GEOMETRY}

\author{\bf Nabil L. Youssef$^{\dagger}$, S. H. Abed$^{\ddagger}$ and A.
Soleiman$^{\flat}$}
\date{}

\maketitle                     % Produces the title.
\vspace{-1.15cm}
\begin{center}
{$^{\dag}$Department of Mathematics, Faculty of Science,\\ Cairo
University, Giza, Egypt}
\end{center}
\vspace{-0.8cm}
\begin{center}
nlyoussef2003@yahoo.fr,\, nyoussef@frcu.eun.eg
\end{center}
\vspace{-0.4cm}

\begin{center}
{$^{\ddag}$Department of Mathematics, Faculty of Science,\\ Cairo
University, Giza, Egypt}
\end{center}
\vspace{-0.8cm}
\begin{center}
sabed52@yahoo.fr,\, sabed@frcu.eun.eg
\end{center}
\vspace{-0.4cm}

\begin{center}
{$^{\flat}$Department of Mathematics, Faculty of Science,\\ Benha
University, Benha,
 Egypt}
\end{center}
\vspace{-0.8cm}

\begin{center}
soleiman@mailer.eun.eg
\end{center}
\smallskip
\vspace{1cm} \maketitle

\smallskip
\bigskip
\noindent{\bf Abstract.}
  The aim of the present paper is to
provide an \emph{intrinsic} investigation  of projective changes  in
  Finlser geometry, following the pullback formalism. Various known
  local results are generalized and  other new intrinsic results
are obtained.
  Nontrivial  characterizations of projective changes are given.
  The fundamental projectively
invariant tensors, namely, the projective deviation tensor,  the
Weyl torsion tensor, the Weyl curvature tensor and the Douglas
tensor are investigated. The properties of these tensors and their
interrelationships are obtained.   Projective connections and
projectively flat manifolds are characterized. The present work is
entirely intrinsic (free from local coordinates).
%\footnote{ArXiv Number: }


\bigskip

\newpage
%Introduction
\vspace{30truept}\centerline{\Large\bf{Introduction}}\vspace{12pt}
\par
\par
The most well-known and widely used approaches to \emph{global}
Finsler geometry are the Klein-Grifone (KG-) approach (\cite{r21},
\cite{r22}, \cite{r27}) and the pullback (PB-) approach (\cite{r58},
\cite{r48}, \cite{r92}). The universe of the first approach is the
tangent bundle of $T M$, whereas the universe of the second is the
pullback of the tangent bundle $TM$ by $\pi: \T M\longrightarrow M$.
Each of the two approaches has its own geometry which differs
significantly from the geometry of the other (in spite of the
existence of some links between them).
\par
The theory of projective changes in Riemannian geometry has been
deeply studied (\emph{locally and intrinsically}) by many authors.
With regard to Finsler geometry, a complete \emph{local} theory of
projective changes has been established (\cite{r104}, \cite{r105},
\cite{r106},~\cite{r42},~\cite{r54},$\cdots$). Moreover, an
intrinsic theory of projective changes (resp. semi-projective
changes) has been investigated in~\cite{r59}, \cite{r103}
(resp.~\cite{r55}) following the KG-approach. To the best of our
knowledge, there is no complete intrinsic theory in the PB-approach.
\par
In this paper we present \emph{an intrinsic} theory of projective
changes in Finsler geometry following \emph{the pullback approach}.
Various known  local results are generalized and other new intrinsic
results are obtained.
\par
The  paper consists of four parts preceded by an introductory
section $(\S 1)$, which provides a brief account of the basic
definitions and concepts necessary for this work.
\par
In the first part $(\S 2)$, the projective change of Barthel
 and  Berwald connections, as well as their curvatuer
tensors, are investigated. Some characterizations of projective
changes are established.  The results obtained in this section play
a key role in obtaining other  results in the next sections.
\par
The second part $(\S 3)$ is devoted to an  investigation of the
fundamental projectively invariant  tensors under a projective
change, namely, the projective deviation tensor,  the Weyl torsion
tensor and the Weyl curvature tensor. The properties of these
tensors and their interrelationships are studied.
\par
The third  part $(\S 4)$ provides a characterization of a linear
connection which is invariant under projective changes (the
projective connection).   Moreover, another fundamental projectively
invariant tensor (the Douglas tensor) is investigated, the
properties of which are discussed. Finally,  the Douglas tensor and
the projective connection are related.
\par
In the fourth  and last part $(\S 5)$, the projective change of some
important special Finsler manifolds, namely, the Berwald, Douglas
and  projectively flat Finsler manifolds are investigated.
Moreover, the relationship between  projectively flat Finsler
manifolds and Douglas tensor (Weyl tensor) is obtained.
\par
It should finally be noted that the present work is entirely
intrinsic (free from local coordinates).

%%%%%%%%%%%%%%%%%%%%%%%% SECTION 1. Notation and Preliminaries %%%%%%%%%%%%%%%%
\Section{Notation and Preliminaries}
In this section, we give a brief account of the basic concepts
 of the pullback approach to intrinsic Finsler geometry necessary for this work. For more
 details, we refer to \cite{r58} and~\,\cite{r48}.
 We make the
assumption that the geometric objects we consider are of class
$C^{\infty}$.  The
following notation will be used throughout this paper:\\
 $M$: a real paracompact differentiable manifold of finite dimension $n$ and of
class $C^{\infty}$,\\
 $\mathfrak{F}(M)$: the $\Real$-algebra of differentiable functions
on $M$,\\
$\mathfrak{X}(M)$: the $\mathfrak{F}(M)$-module of vector fields
on $M$,\\
$\pi_{M}:TM\longrightarrow M$: the tangent bundle of $M$,\\
$\pi: \T M\longrightarrow M$: the subbundle of nonzero vectors
tangent to $M$,\\
$V(TM)$: the vertical subbundle of the bundle $TTM$,\\
 $P:\pi^{-1}(TM)\longrightarrow \T M$ : the pullback of the
tangent bundle $TM$ by $\pi$,\\
 $\mathfrak{X}(\pi (M))$: the $\mathfrak{F}(\T M)$-module of
differentiable sections of  $\pi^{-1}(T M)$,\\
$ i_{X}$ : the interior product with respect to  $X
\in\mathfrak{X}(M)$,\\
$df$ : the exterior derivative  of $f\in \mathfrak{F}(M)$,\\
$ d_{L}:=[i_{L},d]$, $i_{L}$ being the interior derivative with
respect to a vector form $L$.
\par Elements  of  $\mathfrak{X}(\pi (M))$ will be called
$\pi$-vector fields and will be denoted by barred letters
$\overline{X} $. Tensor fields on $\pi^{-1}(TM)$ will be called
$\pi$-tensor fields. The fundamental $\pi$-vector field is the
$\pi$-vector field $\overline{\eta}$ defined by
$\overline{\eta}(u)=(u,u)$ for all $u\in \T M$. We have the
following short exact sequence of vector bundles:\vspace{-0.1cm}
$$0\longrightarrow
 \pi^{-1}(TM)\stackrel{\gamma}\longrightarrow T(\T M)\stackrel{\rho}\longrightarrow
\pi^{-1}(TM)\longrightarrow 0 ,\vspace{-0.1cm}$$
 the bundle morphisms  $\rho$ and $\gamma$ being defined in \cite{r92}. The vector $1$-form
  $J:=\gamma\circ\rho$ is called the natural almost
tangent structure of $T M$. The vertical vector field
$\mathcal{C}:=\gamma\circ\overline{\eta} $ on $TM$ is called  the
canonical (Liouville) vector field.
\par
Let $D$ be  a linear connection (or simply a connection) on the
pullback bundle $\pi^{-1}(TM)$.
 The connection (or the deflection) map associated with $D$ is defined by \vspace{-0.1cm}
$$K:T \T M\longrightarrow \pi^{-1}(TM):X\longmapsto D_X \overline{\eta}
.\vspace{-0.1cm}$$ A tangent vector $X\in T_u (\T M)$ at $u\in \T M$
is horizontal if $K(X)=0$ . The vector space $H_u (\T M)= \{ X \in
T_u (\T M) : K(X)=0 \}$ is called the horizontal space at $u$.
   The connection $D$ is said to be regular if
\begin{equation*}\label{direct sum}
T_u (\T M)=V_u (\T M)\oplus H_u (\T M) \qquad \forall u\in \T M .
\end{equation*}
   Let $\beta:=(\rho |_{H(\T M)})^{-1}$, called the horizontal map of the connection
$D$,  then \vspace{-0.1cm}
   \begin{align*}\label{fh1}
    \rho\circ\beta = id_{\pi^{-1} (TM)}, \quad  \quad
       \beta\circ\rho =   id_{H(\T M)} & {\,\, \text{on}\,\,   H(\T M)}.\vspace{-0.2cm}
\end{align*}

For a regular connection $D$, the horizontal and vertical covariant
derivatives $\stackrel{1}D$ and $\stackrel{2}D$ are defined, for a
vector (1)$\pi$-form $A$, for example, by \vspace{-0.2cm}
 $$ (\stackrel{1}D A)(\overline{X}, \overline{Y}):=
  (D_{\beta \overline{ X}} A)( \overline{Y}), \quad
  (\stackrel{2}D A)( \overline{X},  \overline{Y}):= (D_{\gamma \overline{X}} A)(  \overline{Y}).$$
\par
 The (classical)  torsion tensor $\textbf{T}$  of the connection
$D$ is given by
$$\textbf{T}(X,Y)=D_X \rho Y-D_Y\rho X -\rho [X,Y] \quad
\forall\,X,Y\in \mathfrak{X} (\T M),$$  from which the horizontal
((h)h-) and mixed ((h)hv-) torsion tensors are defined respectively
by \vspace{-0.2cm}
$$Q (\overline{X},\overline{Y}):=\textbf{T}(\beta \overline{X}\beta \overline{Y}),
\, \,\,\,\, T(\overline{X},\overline{Y}):=\textbf{T}(\gamma
\overline{X},\beta \overline{Y}) \quad \forall \,
\overline{X},\overline{Y}\in\mathfrak{X} (\pi (M)).\vspace{-0.2cm}$$
\par
The (classical) curvature tensor  $\textbf{K}$ of the connection $D$
is given by
 $$ \textbf{K}(X,Y)\rho Z=-D_X D_Y \rho Z+D_Y D_X \rho Z+D_{[X,Y]}\rho Z
  \quad \forall\, X,Y, Z \in \mathfrak{X} (\T M),$$
from which the horizontal (h-), mixed (hv-) and vertical (v-)
curvature tensors are defined respectively by
$$R(\overline{X},\overline{Y})\overline{Z}:=\textbf{K}(\beta
\overline{X}\beta \overline{Y})\overline{Z},\quad
P(\overline{X},\overline{Y})\overline{Z}:=\textbf{K}(\beta
\overline{X},\gamma \overline{Y})\overline{Z},\quad
S(\overline{X},\overline{Y})\overline{Z}:=\textbf{K}(\gamma
\overline{X},\gamma \overline{Y})\overline{Z}.$$ The
 (v)h-, (v)hv- and (v)v-torsion tensors  are defined respectively by
$$\widehat{R}(\overline{X},\overline{Y}):={R}(\overline{X},\overline{Y})\overline{\eta},\quad
\widehat{P}(\overline{X},\overline{Y}):={P}(\overline{X},\overline{Y})\overline{\eta},\quad
\widehat{S}(\overline{X},\overline{Y}):={S}(\overline{X},\overline{Y})\overline{\eta}.$$
\par
For a Finsler manifold $(M,L)$, there are canonically associated two
fundamental linear connections, namely, the Cartan connection
$\nabla$ and the Berwald connection $D^{\circ} $. An explicit
expression for the Berwald connection is given by\vspace{-0.2cm}
\begin{thm}\emph{\cite{r92}}\label{pp.3} The Berwald connection $D^{\circ}$ is uniquely determined by
\begin{description}
  \item[(a)] $D^{\circ}_{\gamma \overline{X}} \overline{Y}
  =\rho [\gamma \overline{X}, \beta \overline{Y}]$,\quad \qquad
  \textbf{\emph{(}b\emph{)}} $D^{\circ}_{\beta \overline{X}} \overline{Y}= K [\beta \overline{X}, \gamma \overline{Y}]$.
\end{description}
\end{thm}

%%----------------------------------------------------------------------
%%% \section{Klein-Grifone approach to Finsler geometry}
\par We terminate this section
by some concepts and results concerning the Klein-Grifone approach
to intrinsic Finsler geometry. For more details, we refer to
\cite{r21}, \cite{r22} and  \cite{r27}.

\begin{prop}{\em{\cite{r27}}}\label{pp} Let $(M,L)$ be a Finsler manifold. The vector field
$G$ on $TM$ defined by $i_{G}\,\Omega =-dE$ is a spray, where
 $E:=\frac{1}{2}L^{2}$ is the energy function and $\Omega:=dd_{J}E$.
 Such a spray is called the canonical spray.
 \end{prop}

A nonlinear connection on $M$ is a vector $1$-form $\Gamma$ on $TM$,
$C^{\infty}$ on $\T M$, such that $J \Gamma=J, \,\, \Gamma J=-J .$
The horizontal and vertical projectors
 associated with $\Gamma$ are
defined by
   $h:=\frac{1}{2} (I+\Gamma)$ and $v:=\frac{1}{2}
 (I-\Gamma)$ respectively.   The torsion and  curvature of
$\Gamma$ are defined by $t:=\frac{1}{2}[J,\Gamma]$ and
 $\mathfrak{R}:=-\frac{1}{2}[h,h]$  respectively. A nonlinear
 connection $\Gamma$ is conservative if $d_{h}E=0$.

\begin{thm} \label{th.9a} {\em{\cite{r22}}} On a Finsler manifold $(M,L)$, there exists a unique
conservative homogenous nonlinear  connection  with zero torsion. It
is given by\,{\em:}  $\Gamma = [J,G].$ \\
 Such a nonlinear connection is called the canonical connection, the Barthel
connection or the Cartan nonlinear connection
 associated with $(M,L)$.
\end{thm}

%%%%%%%%%%%%%%%%%%%%%%%%%%%%%%%%%%%%%%% SECTION 2.  %%%%%%%%%%%%%%%%%%%%%%%%%%%%%%%%%
\Section{Projective changes and Berwald connection}

In this section, the projective change of Barthel
 and  Berwald connections, as well as their curvatuer
tensors, are investigated. Some characterizations of projective
changes are established.

\begin{defn}Let $(M,L)$ be a Finsler manifold. A change $L\To \widetilde{L}$ of Finsler structures  is said to be a projective
change if each geodesic of $(M,L)$ is a geodesic of
$(M,\widetilde{L})$ and vice versa. In this case, the Finsler
manifolds $(M,L)$ and $(M,\widetilde{L})$ are said to be
projectively related.
\end{defn}
Let $c:I\longrightarrow M:s\longmapsto(x^{i}(s))\, ; i=1,2,...,n$,
be a regular curve on a Finsler manifold  $(M,L)$. If $G$ is the
canonical spray  associated with $(M,L)$, then $c$ is a geodesic iff
$G(c')=c''$, where $c'=\frac{dc}{ds}$ ($s$ being the arc-length). In
local coordinates, $c$ is a geodesic iff
$$({x^{i}})^{''}+2G^{i}(x,x')=0 .$$
If this equation  is subjected to an arbitrary transformation of its
parameter $s\longmapsto t=t(s)$, with $\frac{dt}{ds}\neq0$, then $c$
is a geodesic iff
$$y^{i}=\frac{dx^{i}}{dt}\,, \quad \frac{dy^{i}}{dt}+2G^{i}(x,y)=\mu y^{i},$$
 where
$\mu:=-(\frac{d^{2}t}{ds^{2}})/(\frac{dt}{ds})^{2}$. It is clear
that these equations remain unchanged if we replace the functions
$G^{i}(x,y)$ by new functions $\widetilde{G^{i}}(x,y)$, the latter
being defined by
\begin{equation}\label{eqq}
    \widetilde{G^{i}}(x,y)=G^{i}(x,y)+\lambda(x,y)y^{i},
\end{equation}
where $\lambda(x,y)$ is an arbitrary function which is  positively
homogenous of degree 1 in the directional argument $y$.  This result
can be formulated intrinsically as follows.
\begin{thm}\label{pp.t3} Two Finsler manifolds  $(M,L)$ and $(M,\widetilde{L})$ are  projectively
related if, and only if,  the associated canonical sprays $G$ and
$\tilde{G}$ are related by \vspace{-0.2cm}
\begin{equation}\label{GG}
  \widetilde{G}= G
-2\lambda(x,y)\mathcal{C} , \vspace{-0.2cm}
\end{equation}
where $\lambda(x,y)$ is a function on $TM$,   positively homogenous
of degree 1 in $y$.
\end{thm}

\begin{rem}It is to be noted  that the local expression of \emph{(\ref{GG})} reduces to \emph{(\ref{eqq})}
and  is in accordance with the existing classical local results on
projective changes \emph{\cite{r42}, \cite{r54}}. For this reason,
we have inserted the factor $(-2)$  in \emph{(\ref{GG})}. Moreover,
this factor facilities calculations.
\end{rem}

In what follows and throughout  we will take (\ref{GG}) as the
definition of a  projective change.
\par
The following lemma is useful for subsequence use.\vspace{-0.2cm}
\begin{lem}\label{pp.t2} The $\pi$-form $\alpha$ defined by

 \begin{equation}\label{alfa}
   \alpha(\overline{X}):=d_{J}\lambda(\beta \overline{X})
 \end{equation}
  has the following
 properties\emph{:}
\begin{description}
  \item[(a)]
   $(D^{\circ}_{\gamma \overline{X}}\alpha )(\overline{Y})=d d_{J}\lambda(\gamma \overline{X}, \beta
   \overline{Y})$,
 \quad\emph{\textbf{(b)}}\  $D^{\circ}_{\gamma \overline{\eta}}\alpha =0$, \quad\emph{\textbf{(c)}}
$\alpha(\overline{\eta})=\lambda$.
\end{description}
\end{lem}
\begin{proof}~\par
 \noindent \textbf{(a)}\ \ We use Theorem  \ref{pp.3} and the
 identities $ \beta\circ\rho +
   \gamma\circ K=I$  and $ i_{J}d_{J}\lambda=0$:
\begin{eqnarray*}
   dd_{J}\lambda(\gamma \overline{X},\beta \overline{Y}) &=& \gamma \overline{X}\cdot d_{J}\lambda(\beta \overline{Y})
   -d_{J}\lambda([\gamma \overline{X},\beta \overline{Y}])  \\
   &=&\gamma \overline{X}\cdot \alpha(\overline{Y})
   -d_{J}\lambda(\beta\rho[\gamma \overline{X},\beta \overline{Y}])\\
     &=&\gamma \overline{X}\cdot \alpha(\overline{Y})
   -\alpha(D^{\circ}_{\gamma \overline{X}}\overline{Y})=(D^{\circ}_{\gamma \overline{X}}\alpha )(\overline{Y}).
\end{eqnarray*}
\vspace{3pt} \noindent \textbf{(b)}\ \ From (a) above, we have
\begin{eqnarray*}
   (D^{\circ}_{\gamma \overline{\eta}}\alpha )(\overline{X})&=&
    \mathcal{C}\cdot d_{J}\lambda(\beta \overline{X})
   -d_{J}\lambda([\mathcal{C},\beta \overline{X}])
   =\mathcal{C}\cdot(\gamma\overline{X}\cdot\lambda)-J[J G,\gamma\overline{X}]\cdot\lambda\\
   &=&\mathcal{C}\cdot(\gamma\overline{X}\cdot\lambda)-\{[J
   G,\gamma\overline{X}]-J[G,\gamma\overline{X}]\}\cdot\lambda, \ \ (as \ J^{2}=0)\\
   &=&\{\mathcal{C}\cdot(\gamma\overline{X}\cdot\lambda)-[J
   G,\gamma\overline{X}]\cdot\lambda\}-JX\cdot\lambda=\gamma\overline{X}\cdot(\mathcal{C}\cdot\lambda)-\gamma\overline{X}\cdot\lambda.
\end{eqnarray*}
As $\mathcal{C}\cdot\lambda=\lambda$, the result follows.
\end{proof}

An important  characterization of projective changes is given by
\vspace{-0.2cm}
\begin{thm}\label{th.3} The following assertions are equivalent:
\begin{description}
  \item[(a)] $(M,L)$ and $(M,\widetilde{L})$  are projectively
related.
  \item[(b)] The associated Barthel connections $\Gamma$ and
$\widetilde{\Gamma}$ are related by
\vspace{-0.2cm}$$\widetilde{\Gamma}=\Gamma - 2\{\lambda J +
d_{J}\lambda \otimes \mathcal{C}\} .$$
  \item[(c)]The associated Berwald connections $D^{\circ}$ and
$\widetilde{D^{\circ}}$ are related by
\begin{equation}\label{eq.1}
     \widetilde{D}^{\circ}_{X}\overline{Y}= D^{\circ} _{X}\overline{Y} +  \alpha( \overline{Y}) \rho X+
\alpha(\rho X) \overline{Y} + (\stackrel{2}{D^{\circ}}\!\alpha)
(\overline{Y},\rho X)\overline{\eta}.
\end{equation}

\end{description}

\end{thm}
\begin{proof}~\par
\noindent\textbf{(a) $\Longrightarrow$(b)}: As $ \widetilde{G}= G
-2\lambda(x,y)\mathcal{C}$, we have
$$
  \widetilde{\Gamma} = [J,\widetilde{G}] \\
   = [J, G -2\lambda \mathcal{C} ] \\
   = [J, G]-2[J,\lambda \mathcal{C}].
$$
Since, for every vector form $A$ on $TM$, $X\in \mathfrak{X}(TM)$
and $f\in\mathfrak{F}(TM)$ \cite{r20},
$$ [A,f X]= f[A, X]+d_{A}f\otimes X-df\wedge i_{X}A ,$$
we obtain
\begin{eqnarray*}
  [J,\lambda \mathcal{C}] &=& \lambda[J, \mathcal{C}]
  +d_{J}\lambda\otimes \mathcal{C}-d\lambda\wedge i_{\mathcal{C}}J \\
  &=& \lambda J +d_{J}\lambda\otimes \mathcal{C},  \ \
  (as \ [J,\mathcal{C}]=J \ and \ i_{\mathcal{C}}J=0).
\end{eqnarray*}
From which (b) follows.

\vspace{6pt}
 \noindent\textbf{(b) $\Longrightarrow$(c)}: If (b) holds, then
\begin{equation}\label{hv}
    \widetilde{h}=h-\lambda J -
d_{J}\lambda \otimes \mathcal{C}, \qquad \widetilde{v}=v+\lambda J +
d_{J}\lambda \otimes \mathcal{C}.
\end{equation}
Using Theorem \ref{pp.3} and (\ref{hv}), we get
\begin{eqnarray}\label{aa}
  \widetilde{D}^{\circ}_{\widetilde{v}X}\rho Y &=&
  \rho[\widetilde{v} X,Y]= \rho[v X,Y]+\rho[\lambda JX+d_{J}\lambda(X)\mathcal{C},Y]\nonumber\\
    & =&\rho[v X,Y]+\{\lambda \rho[ JX,Y]-(Y\cdot \lambda) \rho
    JX\}\nonumber\\
    &&+\{d_{J}\lambda(X)\rho [\mathcal{C},Y]-(Y\cdot d_{J}\lambda(X))\rho(\mathcal{C})\}\nonumber\\
    & =& D^{\circ}_{v X}\rho Y+\lambda \rho[ JX,Y]+\alpha(\rho X)\rho
    [\mathcal{C},Y].
\end{eqnarray}
Similarly, one can show that
\begin{eqnarray*}
\gamma \widetilde{D}^{\circ}_{\widetilde{h}X}\rho Y&=&
\gamma D^{\circ}_{h X}\rho Y-\lambda v\{J[X, JY]+J[JX, Y]\} \\
   & &+(JY \cdot \lambda)
 v(JX)-d_{J}\lambda(X)\{J[\mathcal{C},Y]+J[G,JY]\}\\
 &&+(JY \cdot d_{J}\lambda (X)) v(\mathcal{C})+\lambda
 J[hX,JY]-
 (d_{J}\lambda([JY,hX])) \mathcal{C}.
\end{eqnarray*}
From which, taking into account the fact that $Jv=0$, $vJ=J$ and
 $\gamma:\pi^{-1}(TM)\longrightarrow V(TM)$ is an isomorphism, we get
\begin{eqnarray}\label{pp}
   \widetilde{D}^{\circ}_{\widetilde{h}X}\rho Y
  &=& D^{\circ}_{hX}\rho Y+ d_{J}\lambda(Y)\rho X
   + d_{J}\lambda(X)\rho Y \nonumber \\
 & & + dd_{J}\lambda(JY, hX)\overline{\eta} -\lambda
 \rho[JX,Y]-d_{J}\lambda(X)\rho[\mathcal{C},Y]\nonumber\\
  &=& D^{\circ}_{hX}\rho Y+ \alpha(\rho Y)\rho X
   + \alpha(\rho X)\rho Y + (D^{\circ}_{JY}\alpha)(\rho X)\overline{\eta}\nonumber \\
 & &-\lambda
 \rho[JX,Y]-\alpha(\rho X)\rho[\mathcal{C},Y].
\end{eqnarray}
Hence, (\ref{eq.1}) follows from (\ref{aa}) and (\ref{pp}).

\vspace{6pt}
 \noindent\textbf{(c) $\Longrightarrow$(a)}:  Assume that  Equation (\ref{eq.1}) holds. Then,  by setting
$\overline{Y}=\overline{\eta}$ in (\ref{eq.1}), noting that
$\alpha(\overline{\eta})=\lambda$ (Lemma \ref{pp.t2}(c)), we get
$$\widetilde{K}^{\circ}(X)=K^{\circ}(X)+\lambda \rho X+\alpha(\rho X)\overline{\eta}
+(D^{\circ}_{\gamma \overline{\eta}}\alpha )(\rho
X)\overline{\eta}.$$ From which, together with Lemma \ref{pp.t2}(b)
and the fact that $v=v^{\circ}=\gamma \circ K^{\circ}=\gamma \circ K
$ \cite{r92}, we get

 $$\widetilde{v}X=vX+\lambda JX +\alpha(\rho X) \mathcal{C}.$$
 Consequently,

 \begin{equation*}
 \widetilde{h}X=hX-\lambda JX -\alpha(\rho X) \mathcal{C}.
 \end{equation*}
 Setting $X=G$ in the last relation, taking into account Lemma \ref{pp.t2}(c) and
  $\widetilde{h} G=\widetilde{\beta} \overline{\eta}=\widetilde{G}$,
  we obtain

  $$\widetilde{G}=G-2 \lambda \mathcal{C}.$$

 Hence, by Theorem \ref{pp.t3}, $(M,L)$ and $(M,\widetilde{L})$  are  projectively
related.
\end{proof}

\begin{cor} Under the projective change {\emph{(\ref{GG})}}, the curvature
tensors
 $\mathfrak{R}$ and $\widetilde{\mathfrak{R}}$ of the associated Barthel connections
$\Gamma$ and  $\widetilde{\Gamma}$   are related by
$$\widetilde{\mathfrak{R}}=\mathfrak{R}+ \frac{1}{2}\left(2d_{h}\lambda-d_{J}\lambda^{2}\right)\wedge J
+d_{h}d_{J}\lambda \otimes \mathcal{C}.$$
\end{cor}

\begin{cor}\label{pp.t5}In view of Theorem \ref{th.3}\emph{(c)}, we have\vspace{-0.2cm}
\begin{description}
  \item[(a)] $ \widetilde{D}^{\circ}_{\gamma \overline{X}} \overline{Y}= D^{\circ} _{\gamma
  \overline{X}}\overline{Y}$.
  \item[(b)] $ \widetilde{D}^{\circ}_{\tilde{\beta} \overline{X}} \overline{Y}=D^{\circ}
_{\beta \overline{X}} \overline{Y} +
  \alpha (\overline{Y})\overline{X}+ \alpha(\overline{X})\overline{Y} +
  \stackrel{2}{D^{\circ}}\!\alpha(\overline{Y},\rho X)\overline{\eta}
  -\lambda D^{\circ} _{\gamma \overline{X}} \overline{Y} -\alpha(\overline{X})
  D^{\circ} _{\gamma \overline{\eta}} \overline{Y}$.
\end{description}
Consequently,
\begin{description}
    \item[(c)]  The  map  $D^{\circ}_{\gamma \overline{X}}: \mathfrak{X}(\pi (M)) \To \mathfrak{X} (\pi (M)):
  \overline{Y} \longmapsto D^{\circ}_{\gamma \overline{X}}\overline{Y}$  is a projective invariant.
     \item[(d)] The vector $\pi$-form $D^{\circ}\overline{X}:
     \mathfrak{X}(\pi (M)) \To \mathfrak{X} (\pi (M)):
  \overline{Y} \longmapsto D^{\circ}_{\gamma \overline{Y}}\overline{X}$
  is  a projective invariant.
\end{description}
\end{cor}

\begin{rem} In view of Theorem \emph{\ref{th.3}}, we conclude that, a necessary and sufficient condition for two
Finsler manifolds $(M,L)$ and $(M,\widetilde{L})$  to be
projectively related is that Relation \emph{(\ref{eq.1})} holds.
This result  generalizes the corresponding result on projective
changes in Riemannian geometry. Apart from the last term  of formula
\emph{(\ref{eq.1})}, this formula resembles exactly the
corresponding Riemannian formula \emph{\cite{r83}}. Moreover, the
sufficiency is not proved before, as far as we know. On the other
hand, the local expressions of \emph{\textbf{(a)}} and
\emph{\textbf{(b)}} of Theorem \emph{\ref{th.3}} coincide with the
classical local expressions \emph{\cite{r42}}, \emph{\cite{r54}}.
\end{rem}

 For the projective change of the curvature tensors of
Berwald connection, we need the following two lemmas.
\vspace{-0.2cm}
\begin{lem}\emph{\cite{r96}}\label{pp.10} For the Berwald connection $D^{\circ}$, we have{\,\em:}  \vspace{-0.2cm}
\begin{description}
  \item[(a)]The v-curvature tensor $S^{\circ}$ vanishes.
  \item[(b)]The hv-curvature tensor $P^{\circ}$ is totally symmetric.
\end{description}
\end{lem}

\begin{lem}\label{le.1}~\par\vspace{-0.2cm}
\begin{description}
   \item[(a)]The $\pi$-form $\stackrel{2}{D^{\circ}}\alpha$  is symmetric
   and $(\stackrel{2}{D^{\circ}}\!\alpha)(\overline{\eta},
   \overline{X})=0$.
   \item[(b)]The $\pi$-form $\stackrel{2}{D^{\circ}}\stackrel{2}{D^{\circ}}\alpha$  is totaly symmetric
   and
    $(\stackrel{2}{D^{\circ}}\!\stackrel{2}{D^{\circ}}\alpha)(\overline{X},\overline{Y},\overline{\eta})=
    -(\stackrel{2}{D^{\circ}}\alpha)(\overline{X},\overline{Y})$.
\end{description}
\end{lem}
\begin{proof}~\par
\noindent \textbf{(a)}  By definition of the Berwald vertical
covariant derivative,  we have
\begin{eqnarray*}
  && (\stackrel{2}{D^{\circ}}\alpha)(\overline{X},\overline{Y})
   -(\stackrel{2}{D^{\circ}}\alpha)(\overline{Y},\overline{X})=(D^{\circ}_{\gamma \overline{X}}\alpha)(\overline{Y})- (D^{\circ}_{\gamma \overline{Y}}\alpha)(\overline{X}) \\
 &=& \{\gamma \overline{X}\cdot(d_{J}\lambda(\beta \overline{Y}))
 -d_{J}\lambda(\beta D^{\circ}_{\gamma \overline{X}} \overline{Y})\}-
 \{\gamma \overline{Y}\cdot(d_{J}\lambda(\beta \overline{X}))
 -d_{J}\lambda(\beta D^{\circ}_{\gamma \overline{Y}} \overline{X})\} \\
 &=& \{\gamma \overline{X}\cdot(\gamma \overline{Y}\cdot\lambda)
 -\gamma D^{\circ}_{\gamma \overline{X}} \overline{Y}\cdot\lambda\}-
 \{\gamma \overline{Y}\cdot(\gamma \overline{X}\cdot\lambda)
 -\gamma D^{\circ}_{\gamma \overline{Y}} \overline{X}\cdot\lambda\}\\
 &=&\{[\gamma \overline{X},\gamma \overline{Y}]-\gamma (D^{\circ}_{\gamma \overline{X}} \overline{Y}-
  D^{\circ}_{\gamma \overline{Y}} \overline{X})\}\cdot \lambda  \ \ .
\end{eqnarray*}
 From which, together with the fact that $D^{\circ}$ is torsion free, it follows that $\stackrel{2}{D^{\circ}}\alpha$  is symmetric.
 \par
 On the other hand, $(\stackrel{2}{D^{\circ}}\alpha)(\overline{\eta},
   \overline{X})=0$ is a reformulation of  Lemma \ref{pp.t2}(b).

 \vspace{6pt}
 \noindent \textbf{(b)} By  \textbf{(a)} above and the formula
 \begin{eqnarray*}
   (\stackrel{2}{D^{\circ}}\stackrel{2}{D^{\circ}}\alpha)(\overline{X},\overline{Y},\overline{ Z})&=&
    \gamma \overline{X}\cdot\{(\stackrel{2}{D^{\circ}}\alpha)(\overline{Y},\overline{Z})\}-
    (\stackrel{2}{D^{\circ}}\alpha)(D^{\circ}_{\gamma \overline{X}}\overline{Y},\overline{ Z})-
    (\stackrel{2}{D^{\circ}}\alpha)(\overline{Y},D^{\circ}_{\gamma \overline{X}}\overline{
    Z}),
 \end{eqnarray*}
 it follows that
$\stackrel{2}{D^{\circ}}\stackrel{2}{D^{\circ}}\alpha$ is symmetric
with respect to the second  and the third arguments and
    $(\stackrel{2}{D^{\circ}}\!\stackrel{2}{D^{\circ}}\alpha)(\overline{X},\overline{Y},\overline{\eta})=
    -(\stackrel{2}{D^{\circ}}\alpha)(\overline{X},\overline{Y})$.
  Moreover, one can show that
\begin{eqnarray*}
   \mathfrak{U}_{\overline{X},\overline{Y}}\{(\stackrel{2}{D^{\circ}}\stackrel{2}{D^{\circ}}\alpha)
   (\overline{X},\overline{Y},\overline{ Z})\}&=&
   \{[\gamma \overline{X},\gamma \overline{Y}]-\gamma (D^{\circ}_{\gamma \overline{X}} \overline{Y}-
  D^{\circ}_{\gamma \overline{Y}} \overline{X})\}\cdot \alpha(\overline{Z})\\
  &&+\alpha(-D^{\circ}_{\gamma \overline{X}}
  D^{\circ}_{\gamma\overline{Y}}\overline{Z}+D^{\circ}_{\gamma \overline{Y}}
   D^{\circ}_{\gamma\overline{X}}\overline{Z}+D^{\circ}_{\gamma\{ D^{\circ}_{\gamma \overline{X}}\overline{Y}
   -D^{\circ}_{\gamma \overline{Y}}\overline{X}\}}\overline{Z})\\
  &=&\alpha(S^{\circ}(\overline{X},\overline{Y})\overline{Z})=0, \ \ \text{by Lemma \ref{pp.10}}.\vspace{-0.2cm}
   \end{eqnarray*}
   Hence, $\stackrel{2}{D^{\circ}}\stackrel{2}{D^{\circ}}\alpha$
   is symmetric with respect to the first  and the second arguments.
\end{proof}

\vspace{8pt}
\par Now, let us define \vspace{-0.2cm}
 \begin{equation}\label{Q and E}
\left.
    \begin{array}{rcl}
         Q(\overline{X})&:=&\beta \overline{X}\cdot \lambda-\lambda\,\alpha(\overline{X}),\\
         \varepsilon(\overline{X},\overline{Y})&:=&(D^{\circ}_{\gamma\overline{X}}Q)(\overline{Y})
   -(D^{\circ}_{\gamma\overline{Y}}Q)(\overline{X}).
\end{array}
  \right\}
\end{equation}

Using these $\pi$-tensor fields, we have\vspace{-0.2cm}
\begin{thm}\label{Th.6}   Under the projective change \emph{(\ref{GG})}, we have\vspace{-0.2cm}
\begin{description}
  \item[(a)]$\widetilde{R^{\circ}}(\overline{X},\overline{Y})\overline{Z}
   ={R}^{\circ}(\overline{X},\overline{Y})\overline{Z}+
    (D^{\circ}_{\gamma\overline{Z}}Q)(\overline{Y})\overline{X}
   -(D^{\circ}_{\gamma\overline{Z}}Q)(\overline{X})\overline{Y}+
   \varepsilon(\overline{X},\overline{Y})\overline{Z}
   +(D^{\circ}_{\gamma\overline{Z}}\varepsilon)(\overline{X},\overline{Y})\overline{\eta}$,
  \item[(b)]$ \widetilde{P}^{\circ}(\overline{X},\overline{Y})\overline{Z}
  ={P}^{\circ}(\overline{X},\overline{Y})\overline{Z}+
  \mathfrak{S}_{\overline{X},\overline{Y},\overline{Z}}
  \{((\stackrel{2}{D^{\circ}}\alpha)(\overline{Y},\overline{Z}))\overline{X}\}+
  (\stackrel{2}{D^{\circ}}\stackrel{2}{D^{\circ}}\alpha)(\overline{X},\overline{Y},
  \overline{Z})\overline{\eta}$,
\end{description}
where $\mathfrak{S}_{\overline{X},\overline{Y}, \overline{Z}}$
denotes the cyclic sum over $\overline{X},\overline{Y}$ and
$\overline{Z}$.
\end{thm}
\begin{proof}
 After long, but easy, calculations, these formulae follow by using
 Theorem \ref{th.3}, Lemma \ref{pp.10}, Lemma \ref{le.1} and the properties of the
 $\pi$-forms $\alpha$, $Q$ and $\varepsilon$.
\end{proof}

\begin{cor}\label{cor.1} Under the projective change \emph{(\ref{GG})}, we have
 \begin{description}
   \item[(a)]${\widetilde{\widehat{R}^{\circ}}}(\overline{X},\overline{Y})
   ={\widehat{R}^{\circ}}(\overline{X},\overline{Y})+ Q(\overline{Y})\overline{X}
   -Q(\overline{X})\overline{Y}+\varepsilon(\overline{X},\overline{Y})\overline{\eta}$,
   \item[(b)]$\widetilde{H}(\overline{X})=H(\overline{X}) -Q(\overline{\eta})\overline{X}
   +\{Q(\overline{X})+\varepsilon(\overline{\eta},\overline{X})\}\overline{\eta}$,
    \end{description}
$H$ being the deviation tensor defined by
$H(\overline{X})\!:=\widehat{R}^{\circ}(\overline{\eta},\overline{X})$.
\end{cor}

\begin{prop} Under the projective change \emph{(\ref{GG})},
 if the factor of projectivity $\lambda$ has the property that $
 Q=0$,  then the following geometric objects are projective
   invariants\emph{:}
\vspace{-0.2cm}
\begin{description}
  \item[(a)]The deviation tensor $H$,
  \item[(b)]The (v)h-torsion tensor $\widehat{R^{\circ}},$
  \item[(c)]The (h)h-curvature tensor $R^{\circ},$
  \item[(d)] The curvature tensor $\mathfrak{R}$ of  Barthel connection.
\end{description}
\end{prop}
\begin{proof}
 The proof follows from Theorem \ref{Th.6}(a),  Corollary
\ref{cor.1}, together with the definition of $H$  and the identity
\cite{r96}
  \begin{equation*}
         \mathfrak{R}(\beta
\overline{X},\beta \overline{Y})=-\gamma
\widehat{R^{\circ}}(\overline{X}, \overline{Y}) \vspace{-0.9cm}.
  \end{equation*}
\end{proof}

%%%%%%%%%%%%%%%%%%%%%%%% Section $$$$$$$$$$$$$$$$$$$$$$$$$$
\Section{Weyl  projective tensor}

Studying invariant geometric objects under a given change is of
particular  importance.   In this section,  we investigate
intrinsically the most important  invariant  tensor fields under a
projective change, namely, the projective deviation tensor,  the
Weyl torsion tensor and the Weyl curvature tensor. The properties of
these tensors and their interrelationships are investigated.

\vspace{7pt}
\par In  what follows and throughout, we make use the following convention.
If $A$ is a vector $\pi$-form of degree $3$, for example, we shall
write
$Tr^{c}_{\,\,\,\overline{Z}}\,\{A(\overline{X},\overline{Y},\overline{Z})\}$
to denote  the contracted trace \cite{r59} of $A$ with respect to
$\overline{Z}$:
$Tr^{c}_{\,\,\,\overline{Z}}\,\{A(\overline{X},\overline{Y},\overline{Z})\}:=Tr^{c}\{{\overline{Z}}\longmapsto
A(\overline{X},\overline{Y},\overline{Z})\}$.
\par
 It is to be noted that if a $\pi$-tensor field
$A$ of type (1,p) is projectively  invariant, then  so is its
contracted trace  $Tr^{c}(A)$.
\par
\vspace{7pt}
 Now, let us define\vspace{-0.2cm}
\begin{eqnarray}
  \theta (\overline{X},\overline{Y})&:=& Tr^{c}_{\,\,\,\overline{Z}}
  \,\{R^{\circ}(\overline{X},\overline{Y})\overline{Z}\}, \nonumber \\
  R_{2} (\overline{X},\overline{Y})&:=& Tr^{c}_{\,\,\,\overline{Z}}
  \,\{R^{\circ}(\overline{X},\overline{Z})\overline{Y}\}, \nonumber\\
   R_{1} (\overline{X})&:=&\frac{1}{n-1} \{nR_{2} (\overline{X},\overline{\eta})
   +R_{2} (\overline{\eta},\overline{X})\} \ ; \ n>2, \label{R1}\\
   k&:=&\frac{1}{n-1}R_{2} (\overline{\eta},\overline{\eta})\ ; \ n>2 \label{k}.
\end{eqnarray}
One can show, by using the identity \cite{r96}
$\mathfrak{S}_{\overline{X},\overline{Y},\overline{Z}}
\{R^{\circ}(\overline{X},\overline{Y})\overline{Z}\}=0$, that
\begin{equation}\label{VV}
\theta (\overline{X},\overline{Y})=R_{2}
(\overline{X},\overline{Y})-R_{2} (\overline{Y},\overline{X})
\end{equation}

The following lemma will be useful for subsequent
use.\vspace{-0.2cm}
 \begin{lem}\label{def.1} A $\pi$-tensor field $\omega$ is positively homogenous of degree r in the directional argument $y$
\emph{(}denoted by h\emph{(}r\emph{)}\emph{)} if, and only if
\vspace{-0.1cm}$$D^{\circ}_{\gamma \overline{\eta}}\, \omega= r
\omega, \,\,\text{or equivalently} \,\,\, D^{\circ}_{\mathcal{C}}\,
\omega= r \omega.$$
\end{lem}
In view of the above lemma, we have{\,\em:}
\begin{prop}\label{p.homog.} ~\par \vspace{-0.2cm}
\begin{description}

  \item[(a)]The hv-curvature tensor $P^{\circ}$ is homogenous of degree -1.

    \item[(b)]The h-curvature tensor $R^{\circ}$ is homogenous of degree
    0.
    \item[(c)]The (v)h-torsion tensor $\widehat{R}^{\circ}$ is homogenous of degree
    1.
   \item[(d)]The deviation tensor $H$ is homogenous of degree 2.

 \item[(e)]The $\pi$-tensor fields $R_{2}$ and $\theta $ are homogenous of degree 0.

\item[(f)]The $\pi$-tensor field $R_{1}$ is homogenous of degree 1.

\item[(g)]The scalar function  $k$ is homogenous of degree 2.
\end{description}
\end{prop}

Now, we are in a position to announce the main result of this
section.\vspace{-0.2cm}
\begin{thm}\label{weyl}  Under the projective change \emph{(\ref{GG})}, the following
tensor field on $\p$; $\dim M >2$, is projectively invariant\emph{:}
\begin{eqnarray*}
  W(\overline{X},\overline{Y}) \overline{Z}&:=&
 R^{\circ}(\overline{X},\overline{Y}) \overline{Z}+\frac{1}{n+1}
 \mathfrak{U}_{\overline{X},\overline{Y}}\{(\stackrel{2}{D^{\circ}}R_{1})(\overline{Z},\overline{Y})\overline{X}\\
 &&+ (\stackrel{2}{D^{\circ}}R_{1})(\overline{X},\overline{Y})\overline{Z}+
(\stackrel{2}{D^{\circ}}
\stackrel{2}{D^{\circ}}R_{1})(\overline{Z},\overline{X},\overline{Y})\overline{\eta}\},
\end{eqnarray*}
where\,
$\mathfrak{U}_{\overline{X},\overline{Y}}\{A(\overline{X},\overline{Y})\}:
=A(\overline{X},\overline{Y})-A(\overline{Y},\overline{X})$.
  \end{thm}
\begin{proof}  We have, by Corollary \ref{cor.1}(a),

\begin{equation}\label{EQ.1}
    {\widetilde{{R}^{\circ}}}(\overline{X},\overline{Y})\overline{\eta}
   ={{R}^{\circ}}(\overline{X},\overline{Y})\overline{\eta}+ Q(\overline{Y})\overline{X}
   -Q(\overline{X})\overline{Y}+\varepsilon(\overline{X},\overline{Y})\overline{\eta}.
\end{equation}
Taking the contracted trace of (\ref{EQ.1}) with respect to
$\overline{Y}$, we get
\begin{equation}\label{EQ.2}
    \widetilde{R}_{2}(\overline{X}, \overline{\eta})
   ={R}_{2}(\overline{X},\overline{\eta}) -(n-1)Q(\overline{X})
   +\varepsilon(\overline{X},\overline{\eta}).
\end{equation}
On the other hand, by Theorem \ref{Th.6}(a),
\begin{equation}\label{EQ.3}
    \widetilde{R^{\circ}}(\overline{X},\overline{Y})\overline{Z}
   ={R}^{\circ}(\overline{X},\overline{Y})\overline{Z}+
    (D^{\circ}_{\gamma\overline{Z}}Q)(\overline{Y})\overline{X}
   -(D^{\circ}_{\gamma\overline{Z}}Q)(\overline{X})\overline{Y}+
   \varepsilon(\overline{X},\overline{Y})\overline{Z}
   +(D^{\circ}_{\gamma\overline{Z}}\varepsilon)(\overline{X},\overline{Y})\overline{\eta}.
\end{equation}
Taking the contracted trace of (\ref{EQ.3}) with respect to
$\overline{Z}$,  we obtain
\begin{eqnarray*}
 \widetilde{\theta}(\overline{X},\overline{Y})
   &=&\theta(\overline{X},\overline{Y})+
    (D^{\circ}_{\gamma\overline{X}}Q)(\overline{Y})
   -(D^{\circ}_{\gamma\overline{Y}}Q)(\overline{X})+n
   \varepsilon(\overline{X},\overline{Y})
   +(D^{\circ}_{\gamma\overline{\eta}}\varepsilon)(\overline{X},\overline{Y}).
\end{eqnarray*}
Since the $\pi$-form $\varepsilon$ is h(0), as one can easily show,
the above relation reduces to
\begin{eqnarray*}
 \widetilde{\theta}(\overline{X},\overline{Y})
     &=&\theta(\overline{X},\overline{Y})+(n+1)\varepsilon(\overline{X},\overline{Y}).
\end{eqnarray*}
Consequently, by (\ref{VV}),
\begin{equation*}
\varepsilon(\overline{X},\overline{Y})=\frac{1}{(n+1)}\{\widetilde{R}_{2}(\overline{X},\overline{Y})
-\widetilde{R}_{2}(\overline{Y},\overline{X})\}-\frac{1}{(n+1)}\{R_{2}(\overline{X},\overline{Y})-
R_{2}(\overline{Y},\overline{X})\}.
\end{equation*}
From which,
\begin{equation}\label{EQ.5}
\varepsilon(\overline{X},\overline{\eta})=\frac{1}{(n+1)}\{\widetilde{R}_{2}(\overline{X},\overline{\eta})
-\widetilde{R}_{2}(\overline{\eta},\overline{X})\}-\frac{1}{(n+1)}\{R_{2}(\overline{X},\overline{\eta})-
R_{2}(\overline{\eta},\overline{X})\}.
\end{equation}
Solving (\ref{EQ.2}) and (\ref{EQ.5}) for $Q$, taking (\ref{R1})
into account, we obtain
\begin{equation}\label{EQ.6}
Q(\overline{X})=\frac{1}{(n+1)}\{R_{1}(\overline{X})-\widetilde{R}_{1}(\overline{X})\}.
\end{equation}
This equation, together with (\ref{Q and E}), yield
\begin{equation}\label{EQ.4}
\varepsilon(\overline{X},\overline{Y})=\frac{1}{(n+1)}\mathfrak{U}_{\overline{X},\overline{Y}}
\{(D^{\circ}_{\gamma
\overline{X}}R_{1})(\overline{Y})-(\widetilde{D}^{\circ}_{\gamma
\overline{X}}\widetilde{R}_{1})(\overline{Y})\}.
\end{equation}
Substituting (\ref{EQ.6}) and (\ref{EQ.4}) into  (\ref{EQ.1}), we
get
\begin{eqnarray}\label{h1}
 &&\widetilde{{\widehat{R}}^{\circ}}(\overline{X},\overline{Y})+
   \frac{1}{n+1} \mathfrak{U}_{\overline{X},\overline{Y}}\set{\widetilde{R}_{1}(\overline{Y})\overline{X}+
   (\widetilde{D}^{\circ}_{\gamma
\overline{X}}\widetilde{R}_{1})(\overline{Y})\overline{\eta}}= \nonumber\\
 &&= {\widehat{R}^{\circ}}(\overline{X},\overline{Y})+
   \frac{1}{n+1} \mathfrak{U}_{\overline{X},\overline{Y}}\set{R_{1}(\overline{Y})\overline{X}+
   (D^{\circ}_{\gamma
\overline{X}}R_{1})(\overline{Y})\overline{\eta}}.
\end{eqnarray}
\par
Now, taking the vertical covariant derivative of both sides of
(\ref{h1}) with respect to $\overline{Z}$,  making use of Corollary
\ref{pp.t5} and the identity \cite{r96}
\begin{equation}\label{eq.x2}
{R}^{\circ}(\overline{X},
    \overline{Y})\overline{Z}=
    (D^{\circ}_{\gamma\overline{Z}}\widehat{R^{\circ}})(\overline{X},\overline{Y}),
    \end{equation}
 we get  $\widetilde{W}=W$.
\end{proof}

\begin{defn} The $\pi$-tensor field $W$, defined by Theorem \ref{weyl},
 is called the Weyl curvature tensor.
\end{defn}

In the course of the above proof, we have constructed two other
important projectively invariant tensors  as given by\vspace{-0.2cm}
\begin{thm}\label{weyl1} Under the projective change \emph{(}\ref{GG}\emph{)},  the following
tensor fields on $\p$; $\dim M >2$, are projectively
invariants\emph{:}
\begin{description}
  \item[(a)]$W_{1}(\overline{X}):=H(\overline{X})-k\overline{X} +
   \frac{1}{n+1}\,\{3R_{1}(\overline{X})-(n+1)D^{\circ}_{\gamma\overline{X}}k \}\overline{\eta}$.\\
  This tensor field  is called the projective deviation tensor.
  \item[(b)]$W_{2}(\overline{X},\overline{Y})
   :={\widehat{R}^{\circ}}(\overline{X},\overline{Y})+
   \frac{1}{n+1}\, \mathfrak{U}_{\overline{X},\overline{Y}}\,\{R_{1}(\overline{Y})\overline{X}+
   (D^{\circ}_{\gamma
\overline{X}}R_{1})(\overline{Y})\overline{\eta}\}.$\\
  This tensor field  is called the Weyl torsion tensor.
  \end{description}
\end{thm}

\begin{proof}~\par
\vspace{6pt}
 \noindent \textbf{(a)} Setting $\overline{X}=\overline{\eta}$ into
 (\ref{h1}), we get
 \begin{eqnarray*}
 &&\widetilde{H}(\overline{Y})+
   \frac{1}{n+1}\set{\widetilde{R}_{1}(\overline{Y})\overline{\eta}
   -\widetilde{R}_{1}(\overline{\eta})\overline{Y}+
   (\widetilde{D}^{\circ}_{\gamma\overline{\eta}}\widetilde{R}_{1})(\overline{Y})\overline{\eta}
-(\widetilde{D}^{\circ}_{\gamma\overline{Y}}\widetilde{R}_{1})(\overline{\eta})\overline{\eta}}= \nonumber\\
 &&= H(\overline{Y})+
   \frac{1}{n+1}
   \set{R_{1}(\overline{Y})\overline{\eta}
   -R_{1}(\overline{\eta})\overline{Y}+
   (D^{\circ}_{\gamma\overline{\eta}}R_{1})(\overline{Y})\overline{\eta}
   -(D^{\circ}_{\gamma\overline{Y}}R_{1})(\overline{\eta})\overline{\eta}}.
\end{eqnarray*}
From which, together with  Proposition \ref{p.homog.}(f) and the
identity $R_{1}(\overline{\eta})=(n+1)k$ (by (\ref{R1}) and
(\ref{k})), the result follows.

\vspace{6pt}
 \noindent \textbf{(b)}  Follows from  (\ref{h1}).
\end{proof}

 The next results give some  interesting properties
of the above mentioned  projectively invariant  tensors.
\begin{thm}\label{tt.1}~\par\vspace{-0.1cm}
\begin{description}
  \item[(a)]  The Weyl torsion  tensor $W_{2}$  can be expressed in terms of $W_{1}$ in the form
  $$W_{2}(\overline{X}, \overline{Y})=
    \frac{1}{3}\{(D^{\circ}_{\gamma\overline{X}}W_{1})(\overline{Y})-
    (D^{\circ}_{\gamma\overline{Y}}W_{1})(\overline{X})\}.$$
\item[(b)] The Weyl curvature tensor $W$
 can be expressed in terms of $W_{2}$ in the form
 $$W(\overline{X},\overline{Y})\overline{Z}=
 (D^{\circ}_{\gamma\overline{Z}}W_{2})(\overline{X},\overline{Y}).$$
\end{description}
\end{thm}

\begin{proof}~\par
\vspace{6pt}
 \noindent \textbf{(a)} Follows from   Theorem \ref{weyl1} together with
 the identity \cite{r96}
\begin{equation*}
\widehat{R^{\circ}}(\overline{X}, \overline{Y})=
    \frac{1}{3}\{(D^{\circ}_{\gamma\overline{X}}H)(\overline{Y})-
    (D^{\circ}_{\gamma\overline{Y}}H)(\overline{X})\},\vspace{-0.2cm}
\end{equation*}
taking into account the fact that
$(\stackrel{2}{D^{\circ}}\stackrel{2}{D^{\circ}}k)(\overline{X},
\overline{Y})=(\stackrel{2}{D^{\circ}}\stackrel{2}{D^{\circ}}k)(\overline{Y},
\overline{X})$.

\vspace{6pt}
 \noindent \textbf{(b)} Follows from  Theorem \ref{weyl1}, Theorem \ref{weyl} and
 (\ref{eq.x2}).
\end{proof}

\begin{cor}\label{1tt.1}~\par\vspace{-0.2cm}
\begin{description}
 \item[(a)] The projective deviation tensor $W_{1}$ is $h(2)$ and has the
property that
 \vspace{-0.1cm}$$W_{1}({\overline{\eta}})=0.$$

\item[(b)] The Weyl torsion  tensor $W_{2}$ is $h(1)$ and has the
property that
 \vspace{-0.1cm}$$ W_{2}(\overline{\eta}, \overline{X})=  W_{1}(\overline{X}).$$

\item[(c)] The Weyl curvature tensor $W$   is $h(0)$  and has the
property that
\vspace{-0.1cm}$$ W(\overline{X}, \overline{Y})\overline{\eta}=W_{2}(\overline{X},\overline{Y}).$$
\end{description}
\end{cor}
\begin{proof}~\par
\noindent\textbf{(a)} \ Follows from the homogeneity properties of
$H$, $R_{1}$ and $k$ (Proposition \ref{p.homog.}) together with the
fact that $R_{1}(\overline{\eta})=(n+1)k$ and
$H(\overline{\eta})=0$.
 \vspace{5pt}

 \noindent\textbf{(b)} \ Follows from \textbf{(a)} and Theorem \ref{tt.1}(a).

 \noindent\textbf{(c)} \ Follows from  \textbf{(b)} and  Theorem \ref{tt.1}(b).
\end{proof}

\begin{cor}The following assertion are equivalent\emph{:}
\begin{description}
 \item[(a)] The projective deviation  tensor $W_{1}$ vanishes.
 \item[(b)] The Weyl torsion tensor $W_{2}$ vanishes.
 \item[(c)]The Weyl curvature tensor $W$ vanishes.
\end{description}
\end{cor}

%%%%%%%%%%%%%%%%%%%%%%%%%%%%%%% Section %%%%%%%%%%%%%%%%%%%%%%%%%%%
\Section {Projective connections and Douglas tensor}

In this section, we provide a characterization of a linear
connection which is invariant under a projective change (the
projective connection).   Moreover, as in the previous section, we
investigate intrinsically another fundamental projectively invariant
tensor (the Douglas tensor). Finally, we relate the Douglas tensor
to the projective connection in a natural manner.

\begin{defn} A linear connection on $\p$ is said to be  projective
  if it is  invariant under the projective change
\emph{(\ref{GG})}.
\end{defn}
%A characterization of projective connections is given
%by\vspace{-0.2cm}
\begin{thm}\label{th.6a} A linear connection $\Omega$ on $\p$ is  projective
 if, and only if, it can be expressed in the
form\vspace{-0.2cm}
\begin{equation}\label{pro.con}
  \Omega_{X}\overline{Y}= D^{\circ} _{X} \overline{Y} -\frac{1}{n+1}\{\omega( \overline{Y})\rho X+
\omega(\rho X) \overline{Y} + (p\,(\rho X,\overline{Y}))\overline{\eta}\},\vspace{-0.2cm}
\end{equation}
where \ $\omega(\overline{Y}):= Tr^{c}_{\,\,\,X} \,\{D^{\circ}_{X}
\overline{Y}\}$ and $\,\,\, p\,(\overline{X},\overline{Y}):=
Tr^{c}_{\,\,\,\overline{Z}}\,\{
{P}^{\circ}(\overline{X},\overline{Y})\overline{Z}\}$.
\end{thm}
\begin{proof}
 Under a projective change, we have, by Theorem \ref{th.3},
\begin{equation}\label{eq.6}
  \widetilde{D}^{\circ}_{X}\overline{Y}= D^{\circ} _{X}\overline{Y} +  \alpha(\overline{Y})\rho X+
\alpha(\rho X)\overline{Y} + (\stackrel{2}{D^{\circ}}\alpha )(\rho X,\overline{Y})\overline{\eta},\vspace{-0.2cm}
\end{equation}
Taking the contracted trace of  Equation (\ref{eq.6}) with respect
to $X$, using Lemma \ref{le.1}, we obtain
\begin{equation*}
   \widetilde\omega=\omega +(n+1) \alpha,
\end{equation*}
from which
\begin{equation}\label{eq.7}
 \alpha =\frac{1}{n+1}\{\widetilde\omega-\omega\}.
\end{equation}
On the other hand, by Theorem \ref{Th.6},
\begin{equation}\label{ff}
     \widetilde{P}^{\circ}(\overline{X},\overline{Y})\overline{Z}
  ={P}^{\circ}(\overline{X},\overline{Y})\overline{Z}+
  \mathfrak{S}_{\overline{X},\overline{Y},\overline{Z}}
  \{((\stackrel{2}{D^{\circ}}\alpha)(\overline{Y},\overline{Z}))\overline{X}\}+
  (\stackrel{2}{D^{\circ}}\stackrel{2}{D^{\circ}}\alpha)(\overline{X},\overline{Y}, \overline{Z})\overline{\eta}
\end{equation}
Taking the contracted trace of the above equation with respect to
$\overline{Z}$, using Lemma \ref{le.1}, we get
\begin{equation*}
     \widetilde{p}
  =p+
  (n+1)\stackrel{2}{D^{\circ}}\alpha,
\end{equation*}
from which
\begin{equation}\label{eq.8}
   \stackrel{2}{D^{\circ}}\alpha=\frac{1}{(n+1)}
   \{\widetilde{p}  -p\}.
\end{equation}
Substituting  (\ref{eq.7}) and  (\ref{eq.8}) into (\ref{eq.6}), the
result follows.
\par It should finally be noted that the projective connection $\Omega$ is
uniquely determined by Equation (\ref{pro.con}).
\end{proof}

\begin{rem}
In view of \emph{(\ref{eq.8})}, $\stackrel{2}{D^{\circ}}\alpha=0$
if, and only if, the $\pi$-form $p$ is projectively invariant. In
this case, the formula \emph{(\ref{eq.1})} reduces to
\begin{equation*}
     \widetilde{D}^{\circ}_{X}\overline{Y}= D^{\circ} _{X}\overline{Y} +  \alpha( \overline{Y}) \rho X+
\alpha(\rho X) \overline{Y} ,\vspace{-0.2cm}
\end{equation*}
which has exactly  the same form as the corresponding Riemannian
formula for projective changes.
\end{rem}
\begin{prop}\label{tt} The projective connection $\Omega$ has the
properties\emph{:}
\begin{description}
  \item[(a)] The (classical) torsion associated with $\Omega$
  vanishes.

  \item[(b)] The $v$-curvature tensor associated with $\Omega$
  vanishes.
\end{description}
\end{prop}
\begin{proof}~\\
\noindent \textbf{(a)} As $P^{\circ}$ is totally symmetric (Lemma
\ref{pp.10}(b)), then so is the $\pi$-form $p$. Then, the result
follows from this fact and the expression (\ref{pro.con}).

\noindent\textbf{(b)} For any regular connection $D$ having the
property that $T(\overline{X},\overline{\eta} )=0$, the
$v$-curvature tensor $S$ of $D$ takes the form \cite{r96}:
  \begin{eqnarray*}
  S(\overline{X},\overline{Y})\overline{Z} &=&
 (D_{\gamma \overline{Y}}T)(\overline{X},\overline{Z})-
 (D_{\gamma \overline{X}}T)(\overline{Y},\overline{Z})+T(\overline{X},T(\overline{Y},\overline{Z}))\\
 && -T(\overline{Y},T(\overline{X},\overline{Z}))+T( \widehat{S}(\overline{X},
 \overline{Y}),\overline{Z}).
  \end{eqnarray*}
The result follows from the above relation together with
\textbf{(a)}.
\end{proof}

 \begin{thm}\label{th.6b} Under a projective change,  the $\pi$-tensor field
\vspace{-0.2cm}
\begin{equation}\label{DD}
     \mathbb{P}(\overline{X},\overline{Y})\overline{Z}
  :={P}^{\circ}(\overline{X},\overline{Y})\overline{Z}
  -\frac{1}{n+1}\mathfrak{S}_{\overline{X},\overline{Y},\overline{Z}}
  \{(p\,(\overline{X},\overline{Y}))\overline{Z}\}-\frac{1}{n+1}
  \{(D^{\circ}_{\gamma\overline{Y}}p\,)(\overline{X},\overline{Z})\} \overline{\eta}.
\end{equation}
is invariant.
\end{thm}
\begin{proof} From  Equation (\ref{eq.8}) and Corollary \ref{pp.t5}(c), we have
\begin{equation*}
    \stackrel{2}{D^{\circ}}\alpha=\frac{1}{(n+1)}\{
    \widetilde{p}-p\},\quad
 \stackrel{2}{D^{\circ}} \stackrel{2}{D^{\circ}}\alpha=\frac{1}{(n+1)}\{
    \stackrel{2}{\widetilde{D}^{\circ}}\widetilde{p}\,-\stackrel{2}{D^{\circ}}p\}.
\end{equation*}
From which, together with Theorem \ref{Th.6}(b), the result follows.
\end{proof}

\begin{defn}The $\pi$-tensor field  $\mathbb{P}$ of type $\emph{(1,3)}$ on $\p$ defined by
\emph{(\ref{DD})}
 is called the Douglas tensor associated with the projective change \emph{(\ref{GG})}.
\end{defn}

The following result establishes some important properties of the
Douglas tensor.\vspace{-0.2cm}
\begin{prop}\label{ttt.1}The Douglas tensor $\mathbb{P}$  has the properties\emph{:} \vspace{-0.1cm}
\begin{description}
\item[(a)]$\mathbb{P}$ vanishes if ${P}^{\circ}$ vanishes,

\item[(b)] $\widehat{\mathbb{P}}(\overline{X},\overline{Y})=0$,
\item[(c)]  $\mathbb{P}$ is totally symmetric,
\item[(d)] $(D^{\circ}_{\gamma \overline{X}}\mathbb{P})(\overline{Y},
    \overline{Z},  \overline{W})
    =(D^{\circ}_{\gamma \overline{Z}}\mathbb{P})(\overline{Y},
    \overline{X},  \overline{W})$,

\item[(e)] $\mathbb{P}$ is positively homogenous of degree $-1$ in $y$.
\end{description}
\end{prop}
\begin{proof} The proof is easy and we omit it.
\end{proof}

\begin{rem}It is worth noting that in
projective Riemannian geometry there is only one fundamental
projectively invariant tensor (Weyl curvature tensor), whereas in
projective Finsler geometry there are two fundamental projectively
invariant tensors, one is the Weyl curvature tensor $W$ and the
other is the Douglas tensor $\mathbb{P}$.
\end{rem}

 We terminate this section by the following result which relates
the Douglas  tensor with the projective connection in a natural
manner. This result says roughly that the Douglas  tensor is
completely determined by the projective connection.\vspace{-0.1cm}
\begin{thm}The   Douglas tensor is precisely
the  $hv$-curvature tensor of the projective connection
\emph{(\ref{pro.con})}.
\end{thm}

\begin{proof} Let $\bar{K}$ and $\bar{\beta}$ be the connection map
and the horizontal map of the projective connection $\Omega$
respectively.
 By Theorem \ref{th.6a} and  the fact that
 $p\,(\overline{X},\overline{\eta})=0$, the connection map $\bar{K}$
  takes the form
\begin{equation*}
  \bar{K}(X)= K(X) -\frac{1}{n+1}\{\omega( \overline{\eta})\rho X+
\omega(\rho X) \overline{\eta}\}.
\end{equation*}
From which,   together with   Proposition \ref{tt}(a)  and
Proposition 2.2 of \cite{r96}, the horizontal  map $\bar{\beta}$  is
given by
\begin{equation}\label{K}
  \bar{\beta}(\overline{X})= \beta(\overline{X}) +\frac{1}{n+1}\{\omega( \overline{\eta})\gamma \overline{X}+
\omega( \overline{X}) \gamma\overline{\eta}\}.
\end{equation}

Now,  after somewhat long  calculations, using (\ref{K}) and Theorem
 \ref{pp.3}, we get
\begin{eqnarray*}
 \Omega_{\bar{\beta} \overline{X}}\Omega_{\gamma \overline{Y}}\overline{Z}
 &=& D^{\circ} _{\beta \overline{X}}D^{\circ}_{\gamma \overline{Y}}\overline{Z}
    -\frac{1}{n+1}\{\omega(D^{\circ}_{\gamma \overline{Y}}\overline{Z})\overline{X}+
\omega(\overline{X})D^{\circ}_{\gamma \overline{Y}}\overline{Z} +
(p\,(\overline{X},D^{\circ}_{\gamma
\overline{Y}}\overline{Z}))\overline{\eta}\}\\
 &&+\frac{1}{n+1}\{\omega( \overline{\eta})D^{\circ}_{\gamma \overline{X}}D^{\circ}_{\gamma \overline{Y}}
 \overline{Z}+
\omega( \overline{X}) D^{\circ}_{\gamma
\overline{\eta}}D^{\circ}_{\gamma \overline{Y}}\overline{Z}\},
\end{eqnarray*}

\begin{eqnarray*}
 \Omega_{\gamma \overline{Y}}\Omega_{\bar{\beta} \overline{X}}\overline{Z}
     &=& D^{\circ}_{\gamma \overline{Y}}D^{\circ} _{\beta \overline{X}}\overline{Z} -\frac{1}{n+1}
    \{(D^{\circ}_{\gamma \overline{Y}}\omega(\overline{X}))\overline{Z}+
    \omega(\overline{X})D^{\circ}_{\gamma
    \overline{Y}}\overline{Z}+\\
&&+ (D^{\circ}_{\gamma \overline{Y}}\omega(\overline{Z}))
\overline{X}+ \omega(\overline{Z}) D^{\circ}_{\gamma
\overline{Y}}\overline{X}\} -\frac{1}{n+1}\{( D^{\circ}_{\gamma
\overline{Y}}p\,( \overline{X},
\overline{Z}))\overline{\eta}+(p\,(\overline{X},
\overline{Z}))\overline{Y}\}\\
&&+\frac{1}{n+1}\{(D^{\circ}_{\gamma \overline{Y}}\omega(
\overline{\eta}))D^{\circ}_{\gamma \overline{X}}
 \overline{Z}+
( D^{\circ}_{\gamma \overline{Y}}\omega(
\overline{X}))D^{\circ}_{\gamma
\overline{\eta}}\overline{Z}\}\\
 &&+\frac{1}{n+1}\{\omega( \overline{\eta})D^{\circ}_{\gamma \overline{Y}}D^{\circ}_{\gamma \overline{X}}
 \overline{Z}+
\omega( \overline{X}) D^{\circ}_{\gamma
\overline{Y}}D^{\circ}_{\gamma \overline{\eta}}\overline{Z}\}
\end{eqnarray*}
and
\begin{eqnarray*}
 \Omega_{[\bar{\beta} \overline{X},\gamma \overline{Y}]}\overline{Z}
&=&D^{\circ} _{[\beta \overline{X},\gamma \overline{Y}]}\overline{Z}
 +\frac{1}{n+1}\{\omega(\overline{Z})D^{\circ}_{\gamma \overline{Y}}\overline{X}
 +\omega(D^{\circ}_{\gamma \overline{Y}}\overline{X})\overline{Z} +
(p\,(D^{\circ}_{\gamma
\overline{Y}}\overline{X},\overline{Z}))\overline{\eta}\}\\
 &&+\frac{1}{n+1}\{\omega( \overline{\eta})D^{\circ}_{[\gamma \overline{X},\gamma
 \overline{Y}]} \overline{Z}+
\omega( \overline{X}) D^{\circ}_{[\gamma \overline{\eta},\gamma
 \overline{X}]} \overline{Z}\}\\
&&-\frac{1}{n+1}\{(D^{\circ}_{\gamma \overline{Y}}\omega(
\overline{\eta}))D^{\circ}_{\gamma \overline{X}}
 \overline{Z}+
( D^{\circ}_{\gamma \overline{Y}} \omega(
\overline{X}))D^{\circ}_{\gamma \overline{\eta}}\overline{Z}\}.
\end{eqnarray*}
  The result follows from the above three relations together with Lemma
  \ref{pp.10}(a) and the fact that
  $p=\stackrel{2}{D^{\circ}}\omega$.
 \end{proof}

%%%%%%%%%%%%%%%%%%%%%%%%%%%%%%% Section %%%%%%%%%%%%%%%%%%%%%%%%%%%
\Section{Projectively flat manifolds}

 \par
 In this section we investigate  intrinsically  the projective change of some
important special Finsler manifolds, namely, the Berwald, Douglas
and  projectively flat Finsler manifolds.
 Moreover, the relationship between  projectively flat Finsler manifolds and
the Douglas tensor (Weyl tensor) is obtained.

\begin{defn}\label{df.1} A Finsler manifold $(M,L)$ is said to be
a Berwald manifold if the $(h)hv$-torsion tensor $T$ of the Cartan
connection $\nabla$ is horizontally
  parallel. That is,
  $$\nabla_{\beta \overline{X}}\,T =0.$$
\end{defn}
\begin{defn}\label{df.2} A Finsler manifold $(M,L)$ is said to be
a Douglas manifold if its Douglas tensor  vanishes identically:
$\mathbb{P} =0.$
\end{defn}
\begin{thm}\label{th.a11} Every Berwald manifold  is a Douglas manifold.
\end{thm}
\begin{proof} We first show that the $hv$-curvature  $P^{\circ}$  of  a Berwald manifold vanishes. \\
 The $hv$-curvature $P$ and the $(v)hv$-torsion of the
  Cartan connection can be put respectively in the form \cite{r96}:
\begin{eqnarray*}
 P(\overline{X},\overline{Y},\overline{Z}, \overline{W}) &=&
   g(( \nabla_{\beta \overline{Z}}T)(\overline{X}, \overline{Y}),\overline{W})
   -g(( \nabla_{\beta \overline{W}}T)(\overline{X},
   \overline{Y}),\overline{Z})\\
   &&+ g(T(\overline{X},\overline{Z}),\widehat{P}(\overline{W},\overline{Y}))
   -g(T(\overline{X},\overline{W}),\widehat{P}(\overline{Z},
   \overline{Y})),
\end{eqnarray*}
$$\widehat{P}(\overline{X},\overline{Y})=(\nabla_{\beta \overline{X}}\,T)(\overline{X},\overline{Y}).$$
But since  $\nabla_{\beta \overline{X}}\,T =0$, then both
$\widehat{P}$ and $P$ vanish.
\par
On the other hand, the $hv$-curvatures $P^{\circ}$ and
 $P$ are related by \cite{r96}:
\begin{eqnarray*}
 {{P}}^{\circ}(\overline{X},\overline{Y})\overline{Z}&=&
P(\overline{X},\overline{Y})\overline{Z}+
  (\nabla_{\gamma \overline{Y}}\widehat{P})(\overline{X},\overline{Z})
  +\widehat{P}(T(\overline{Y},\overline{X}),\overline{Z})
  +\widehat{P}(\overline{X},T(\overline{Y},\overline{Z}))\\
&&   +(\nabla_{\beta\overline{X}}T)(\overline{Y},\overline{Z})-
   T(\overline{Y}, \widehat{P}(\overline{X},\overline{Z}))-T(
\widehat{P}(\overline{X}, \overline{Y}), \overline{Z}).
\end{eqnarray*}
From which, together with $\nabla_{\beta \overline{X}}\,T
=\widehat{P}=P=0$, it follows that $ {P}^{\circ}=0$.
\par
Now, the vanishing of $ {P}^{\circ}$ implies that $p=0$.
 Hence, by Theorem
\ref{th.6b}, the Douglas tensor $\mathbb{P}$ vanishes identically.
\end{proof}
As a consequence of Theorem \ref{th.a11}, we retrieve intrinsically
a result of Matsumoto \cite{r105}:
\begin{cor}\label{cor.55}
 A Finsler manifold which is projective to a Berwald manifold is a Douglas manifold.
\end{cor}
\begin{proof}
 Since $(M,\widetilde{L})$ is a Berwald manifold, then,
 its $hv$-curvature $\widetilde{P^{\circ}}$
vanishes and consequently $\widetilde{p\,}=0$. Hence, the Douglas
tensor $\mathbb{P}$, which coincides with $\widetilde{\mathbb{P}}$
(by Theorem \ref{th.6b}), vanishes identically. \
\end{proof}
Now, we focus our attention  on  projectively flat Finsler
manifolds.\vspace{-0.2cm}
\begin{defn}\label{hv.h}Under a projective change, a Finsler manifold $(M,L)$ is said to be:\\
--  $hv$-projectively flat  if the $hv$-curvature tensor
$\widetilde{P^{\circ}}$  vanishes.\\
-- $h$-projectively flat  if the $h$-curvature tensor
$\widetilde{R^{\circ}}$  vanishes.\\
-- projectively flat  if both $\widetilde{P^{\circ}}$ and
$\widetilde{R^{\circ}}$  vanish.
\end{defn}

In view of the proof of Corollary \ref{cor.55} and Definition
\ref{hv.h}, we have\vspace{-0.2cm}
\begin{cor} Under the  projective
change \emph{(\ref{GG})}, if $(M,L)$ is  $hv$-projectively flat,
then it is a Douglas manifold.
\end{cor}

The converse is also true under a certain condition: Consider a
projective change for which the projective factor $\lambda$ (or
$\alpha$) satisfies the condition
\begin{equation} \label{eq.x4}
    \stackrel{2}{D^{\circ}}\alpha=-\frac{1}{(n+1)}\, p.
\end{equation}
From which, we get
\begin{equation}\label{EQ.12p}
  \stackrel{2}{D^{\circ}}\stackrel{2}{D^{\circ}}\alpha=-\frac{1}{(n+1)}
   \stackrel{2}{D^{\circ}}p.
\end{equation}

On the other hand, if  the Douglas  tensor  $\mathbb{P}$ vanishes,
then, by Theorem \ref{th.6b}, we obtain
$${P}^{\circ}(\overline{X},\overline{Y})\overline{Z}
  =\frac{1}{n+1}\mathfrak{S}_{\overline{X},\overline{Y},\overline{Z}}
  \{(p\,(\overline{X},\overline{Y}))\overline{Z}\}+\frac{1}{n+1}
  \{(\stackrel{2}{D^{\circ}}p)(\overline{Y},\overline{X}, \overline{Z})\} \overline{\eta}.$$
    From which,
together with (\ref{ff}) and Lemma \ref{le.1}, the $hv$-curvature
$\widetilde{P}^{\circ}$ has the form
\begin{eqnarray*}
\widetilde{P}^{\circ}(\overline{X},\overline{Y})\overline{Z}
  &=&\frac{1}{n+1}\mathfrak{S}_{\overline{X},\overline{Y},\overline{Z}}
  \{(p\,(\overline{X},\overline{Y}))\overline{Z}\}+\frac{1}{n+1}
  \{(\stackrel{2}{D^{\circ}}p)(\overline{Y},\overline{X}, \overline{Z})\}
  \overline{\eta}\\
  &&+  \mathfrak{S}_{\overline{X},\overline{Y},\overline{Z}}
  \{((\stackrel{2}{D^{\circ}}\alpha)(\overline{Y},\overline{Z}))\overline{X}\}+
  (\stackrel{2}{D^{\circ}}\stackrel{2}{D^{\circ}}\alpha)(\overline{X},\overline{Y}, \overline{Z})\overline{\eta}\\
  &=&\mathfrak{S}_{\overline{X},\overline{Y},\overline{Z}}
  \{\{(\stackrel{2}{D^{\circ}}\alpha)(\overline{Y},\overline{Z})+\frac{1}{(n+1)}
   p(\overline{Y},\overline{Z})\}\overline{X}\}\\
   &&\{(\stackrel{2}{D^{\circ}}\stackrel{2}{D^{\circ}}\alpha)(\overline{Y},\overline{X}, \overline{Z})
   +\frac{1}{(n+1)}
   (\stackrel{2}{D^{\circ}}p)(\overline{Y},\overline{X}, \overline{Z})\}\overline{\eta}.
\end{eqnarray*}
\par
Now, using (\ref{eq.x4}) and (\ref{EQ.12p}),
${\widetilde{{P}^{\circ}}}$ vanishes.  Hence, we have

\begin{thm}\label{tth.1} Under a  projective change satisfying
\emph{(\ref{eq.x4})}, a Finsler manifold $(M,L)$ is
$hv$-projectively flat if, and only if, $(M,L)$ is a Douglas
manifold.
\end{thm}

 Now, we study   $h$-projectively flat Finsler manifolds\vspace{-0.2cm}
\begin{prop} Under the  projective
change \emph{(\ref{GG})}, if $(M,L)$; $  \dim M>2$,  is
$h$-projectively flat, then its  Weyl torsion tensor vanishes .
\end{prop}
\begin{proof}Since $(M,L)$ is  $h$-projectively flat,
then, by Definition \ref{hv.h}, $\widetilde{R^{\circ}}$ vanishes and
hence $\widetilde{R}_{2}=\widetilde{R}_{1}=0$. Consequently, by
Theorem \ref{weyl1}(b), $\widetilde{W}_{2}$ vanishes. As the Weyl
torsion tensor $W_{2}$ is invariant, then $W_{2}$ also vanishes.
\end{proof}

The converse is also true under a certain condition: Consider a
projective change for which the projective factor $\lambda$ (or
$\alpha$) satisfies the condition
\begin{equation}\label{EQ.7}
Q(\overline{X})=\frac{1}{(n+1)}R_{1}(\overline{X}).
\end{equation}
This condition implies that
\begin{equation}\label{EQ.10}
\varepsilon(\overline{X},\overline{Y})=\frac{1}{(n+1)}\{(D^{\circ}_{\gamma
\overline{X}}R_{1})(\overline{Y})-({D}^{\circ}_{\gamma
\overline{Y}}{R}_{1})(\overline{X})\}.
\end{equation}

On the other hand, if  the Weyl torsion tensor  $W_{2}$ vanishes,
then, by Theorem \ref{weyl1}(b), we obtain
$${\widehat{R}^{\circ}}(\overline{X},\overline{Y})=-
   \frac{1}{n+1} \mathfrak{U}_{\overline{X},\overline{Y}}\set{R_{1}(\overline{Y})\overline{X}+
   (D^{\circ}_{\gamma \overline{X}}R_{1})(\overline{Y})\overline{\eta}}.$$
    From which,
together with (\ref{EQ.1}), we have
\begin{eqnarray*}
 {\widetilde{\widehat{R}^{\circ}}}(\overline{X},\overline{Y})&=&
  -\frac{1}{n+1} \mathfrak{U}_{\overline{X},\overline{Y}}\set{R_{1}(\overline{Y})\overline{X}+
   (D^{\circ}_{\gamma \overline{X}}R_{1})(\overline{Y})\overline{\eta}} \\
   &&+ Q(\overline{Y})\overline{X}
   -Q(\overline{X})\overline{Y}+\varepsilon(\overline{X},\overline{Y})\overline{\eta}\\
   &=&\{Q(\overline{Y})-\frac{1}{(n+1)}R_{1}(\overline{Y})\}\overline{X}-
   \{Q(\overline{X})-\frac{1}{(n+1)}R_{1}(\overline{X})\}\overline{Y}\\
   &&+\{ \varepsilon(\overline{X},\overline{Y})-\frac{1}{(n+1)}\{(D^{\circ}_{\gamma
\overline{X}}R_{1})(\overline{Y})-({D}^{\circ}_{\gamma
\overline{Y}}{R}_{1})(\overline{X})\}\}\overline{\eta}.
\end{eqnarray*}
\par
Now, using (\ref{EQ.7}) and (\ref{EQ.10}),
${\widetilde{\widehat{R}^{\circ}}}$ vanishes. Consequently,
${\widetilde{R^{\circ}}}$ vanishes, by (\ref{eq.x2}).
 Hence, we have
\begin{thm}\label{tth.3} Under a  projective change satisfying
\emph{(\ref{EQ.7})}, a Finsler manifold $(M,L)$; $  \dim M>2$, is
$h$-projectively flat if, and only if, its Wyel torsion tensor
vanishes.
\end{thm}

Combining Theorem \ref{tth.1} and  Theorem \ref{tth.3}, we retrieve
intrinsically a result of Matsumoto \cite{r105}. Namely, we have
\begin{thm}\label{th.q} Under a  projective change satisfying
\emph{(\ref{eq.x4})} and \emph{(\ref{EQ.7})}, a Finsler manifold
$(M,L)$; $  \dim M>2$, is projectively flat if, and only if, its
Wyel torsion tensor
 and  Douglas tensor vanish.
\end{thm}

%%%%%%%%%%%%%%%%%%% The references %%%%%%%%%%%%%%%%%%%%%%%%%%%%%%%%%%%%%

\providecommand{\bysame}{\leavevmode\hbox to3em{\hrulefill}\thinspace}
\providecommand{\MR}{\relax\ifhmode\unskip\space\fi MR }
\providecommand{\MRhref}[2]{%
  \href{http://www.ams.org/mathscinet-getitem?mr=#1}{#2}
}
\providecommand{\href}[2]{#2}


\begin{thebibliography}{10}

\bibitem{r58}
H.~Akbar-Zadeh, \emph{Initiation to global Finsler geometry},
Elsevier, 2006.

\bibitem{r101}
P.~L.~Antonelli, R.~Ingarden and M.~Matsumoto, \emph{The theory of
sprays and Finsler spaces with applications in physics and biology},
Kluwer Acad. publ., Netherlands, 1993.


\bibitem{r59}
L. R.~del Castillo, \emph{Tenseurs de {W}eyl d'une gerbe de directions}, C.
  R. Acad. Sc. Paris, Ser. A, \textbf{{282}} (1976), 595--598.

\bibitem{r104}
M. Crampin and  D. J. Saunders,   \emph{Affine and projective
transformations of Berwald connections}, Diff. Geom. Appl.,
\textbf{25} (2007), 235–-250.


%\bibitem{r100}
%Byung-Doo Kim* and Ha-Yong Park, \emph{On special Finsler spaces
%with common geodesics}, Comm. Korean Math. Soc., 15, 2 (2000),
%331--338

\bibitem{r20}
A. Fr\"{o}licher and A. Nijenhuis, \emph{Theory of vector-valued
differential forms}, {I}, Ann. Proc. Kon. Ned. Akad., A,
\textbf{{59}} (1956), 338-359.

\bibitem{r21}
J.~Grifone, \emph{Structure pr\'esque-tangente et connexions},
 I, Ann.  Inst. Fourier, Grenoble, \textbf{22}, {1} (1972), 287-334.

\bibitem{r22}
J.~Grifone, \emph{Structure presque-tangente et connexions},
 II, Ann.   Inst. Fourier, Grenoble, \textbf{22}, {3} (1972), 291-338.

\bibitem{r27}
J.~Klein and A.~Voutier, \emph{Formes ext\'{e}rieures
g\'{e}n\'{e}ratrices de  sprays}, Ann. Inst. Fourier, Grenoble,
\textbf{18}, 1 (1968), 241-260.

\bibitem{r105}
M. Matsumoto, \emph{Projective changes of Finsler metrics and
projectively flat Finsler spaces}, Tensor, N. S., \textbf{34}
(1980), 303--315.

\bibitem{r106}
 M. Matsumoto, \emph{Projectively flat Finsler spaces with $(\alpha ,
\beta)$-metric}, Rep. Math. Phys., \textbf{30} (1991), 15--20.


\bibitem{r42}
H.~Rund, \emph{The differential geometry of {F}insler spaces},
  Springer-Verlag, Berlin, 1959.

\bibitem{r103}
J. Szilasi and Sz. Vattam\'{a}ny,  \emph{On the projective geometry
of sprays}, Diff. Geom. Appl., \textbf{12} (2000), 185–-206.

\bibitem{r48}
A.~A. Tamim, \emph{Special {F}insler manifolds},  J. Egypt. Math.
Soc., \textbf{10}, 2 (2002), 149--177.


\bibitem{r54}
T.~Yamada, \emph{On projective changes in {F}insler spaces}, Tensor, N.
  S., \textbf{{52}} (1993), 189--198.


\bibitem{r83}
 K. Yano,    \emph{ Integral {F}ormulas in \textsc{R}iemannian geometry},
  Marcel Dekker Inc., New York, 1970.


\bibitem{r55}
Nabil L. Youssef, \emph{Semi-projective changes}, Tensor, N. S.,
\textbf{55} (1994), 131-141.

\bibitem{r92}
Nabil~L. Youssef, S.~H. Abed and A.~Soleiman, \emph{Cartan and
  Berwald connections in the pullback formalism}, To appear in: Algebras,
  Groups and Geometries, \textbf{25}, 3 (2008). ArXiv Number: 0707.1320.

\bibitem{r96}
Nabil~L. Youssef, S.~H. Abed and A.~Soleiman, \emph{Geometric
objects associated with the fundumental connections in Finsler
geometry},
  To appear in:  J. Egypt. Math. Soc., \textbf{17} (2009). ArXiv Number: 0805.2489.
\end{thebibliography}
\end{document}